\documentclass[11pt]{amsart}

\usepackage {amsfonts,amscd,amssymb,amsmath}
\usepackage{enumerate}

\newtheorem {proposition}{Proposition}[subsection]
\newtheorem {lemma}[proposition]{Lemma}
\newtheorem {theorem}[proposition]{Theorem}
\newtheorem {corollary}[proposition]{Corollary}

\theoremstyle{definition}

\newtheorem{remark}[proposition]{Remark}
\newtheorem{question}[proposition]{Question}

\newcommand{\Gothicd}{\mathfrak{h}_{\Gamma}^{*}}
\newcommand{\Gothice}{\mathfrak{h}_{\Gamma}^{\perp}}

\newenvironment{Dynkin}{\setlength{\unitlength}{1.5pt}\begin{array}{l}}{\end{array}}
\newcommand{\Dbloc}[1]{\begin{picture}(20,20)#1\end{picture}}
\newcommand{\Dcirc}{\put(10,10){\circle{4}}}

\newcommand{\Deast}{\put(12,10){\line(1,0){8}}}
\newcommand{\Dwest}{\put(8,10){\line(-1,0){8}}}
\newcommand{\Dnorth}{\put(10,12){\line(0,1){8}}}
\newcommand{\Dsouth}{\put(10,8){\line(0,-1){8}}}

\newcommand{\Ddoubleeast}{\put(10,12){\line(1,0){10}}\put(10,8){\line(1,0){10}}}
\newcommand{\Ddoublewest}{\put(10,12){\line(-1,0){10}}\put(10,8){\line(-1,0){10}}}
\newcommand{\Ddots}{\put(5,10){\circle*{0.9}}\put(10,10){\circle*{0.9}}\put(15,10){\circle*{0.9}}}

\newcommand{\Dtext}[2]{\makebox(20,20)[#1]{\scriptsize $#2$}}

\newcommand{\Dskip}{\\ [-4.5pt]}
\newcommand{\Dspace}{\Dbloc{}}

\renewcommand{\subsection}[1]%
{\refstepcounter{subsection}%
{\par\noindent\bf\thesubsection.~~}}

\begin{document}

\title[Affine slice for the coadjoint action]%
{Affine slice for the coadjoint action of a class of biparabolic subalgebras of a semisimple Lie algebra}

\author{Patrice Tauvel}
\address{UMR 6086 CNRS, D\'epartement de Math\'ematiques, Universit\'e de Poitiers, 
T\'el\'eport 2 - BP 30179, Boulevard Marie et Pierre Curie, 86962 Futuroscope Chasseneuil cedex, France.}
\email{Patrice.Tauvel@math.univ-poitiers.fr}

\author{Rupert W.T. Yu}
\address{UMR 6086 CNRS, D\'epartement de Math\'ematiques, Universit\'e de Poitiers, 
T\'el\'eport 2 - BP 30179, Boulevard Marie et Pierre Curie, 86962 Futuroscope Chasseneuil cedex, France.}
\email{Rupert.Yu@math.univ-poitiers.fr}

\begin{abstract}
In this article, we give a simple explicit construction of an affine slice for 
the coadjoint action of a certain class of biparabolic (also called seaweed) subalgebras of 
a semisimple Lie algebra over an algebraically closed field of characteristic zero. 
In particular, this class includes all Borel subalgebras. 
\end{abstract}

\maketitle

\section{Introduction}\label{sec:intro}

\subsection{}\label{definition}
Throughout this paper, $\Bbbk$ is an algebraically closed field of characteristic zero, 
All vector spaces and Lie algebras considered are defined over $\Bbbk$. We consider 
the Zariski topology on these spaces. If $X$ is an algebraic variety and $x\in X$, we 
denote by $\mathrm{T}_{x}(X)$ the tangent space of $X$ at $x$.

Let $\mathfrak{g}$ be a finite-dimensional Lie
algebra over $\Bbbk$ and $G$ its algebraic adjoint group.
Recall that $G$ and $\mathfrak{g}$ act naturally on $\mathfrak{g}^{*}$ via the coadjoint action. 
More precisely, for $f\in \mathfrak{g}^{*}$, $\sigma\in G$ and $X,Y\in\mathfrak{g}$, we have
$$
(\sigma.f) (Y) = f(\sigma^{-1}(Y)) \ , \  (X.f)(Y) = f([Y,X]).
$$

An {\bf affine slice} for the coadjoint action of $\mathfrak{g}$  
is an affine subspace $V$ of $\mathfrak{g}^{*}$ such that there exists an open subset $\mathcal{O}$
of $V$ verifying the following conditions~:
\begin{itemize}
\item[$(\mathrm{C}_{1})$] The set $G.\mathcal{O}$ is dense in $\mathfrak{g}^{*}$.
\item[$(\mathrm{C}_{2})$] For all $f\in \mathcal{O}$, we have $\mathrm{T}_{f}(G.f) \cap \mathrm{T}_{f}(\mathcal{O}) = \{0\}$
\item[$(\mathrm{C}_{3})$] For all $f\in \mathcal{O}$, we have $G.f\cap \mathcal{O} = \{ f\}$.
\end{itemize}

An affine slice may not exist, but when it does, we can deduce (using Rosenlicht's Theorem for example, see 
Theorem \ref{extension}) that the field of $G$-invariant rational functions on $\mathfrak{g}^{*}$ is a purely transcendental 
extension of the ground field $\Bbbk$.

If $\mathfrak{g}$ is abelian, then $\mathfrak{g}^{*}$ is an affine slice.
On the other hand, when $\mathfrak{g}$ is a semisimple Lie algebra,  we may identify $\mathfrak{g}$ with
$\mathfrak{g}^{*}$ via the Killing form, and such a slice has been constructed by Kostant \cite{Kos1}
by using a principal $S$-triple. Kostant \cite{Kos2} also constructed an affine slice for the nilpotent radical of a Borel subalgebra 
of a semisimple Lie algebra. 

\subsection{}\label{question}
Let us assume from now on that $\mathfrak{g}$ is a semisimple Lie algebra. A Lie subalgebra $\mathfrak{q}$ 
of $\mathfrak{g}$ is a {\it biparabolic subalgebra} or {\it seaweed subalgebra} 
if there exists a pair of parabolic subalgebras $(\mathfrak{p},\mathfrak{p}')$
such that $\mathfrak{q} = \mathfrak{p} \cap \mathfrak{p}'$ and $\mathfrak{p}+\mathfrak{p}'=\mathfrak{g}$ 
(such a pair of parabolic subalgebras is called {\it weakly opposite}, see \cite{Panyushev} or 
\cite[Chapter 40]{TYbook}). 

We are interested in the following question :

\begin{question}
Does an affine slice exist for the coadjoint action of $\mathfrak{q}$ ?
\end{question}

Motivated by the study of semi-invariant polynomials on the dual, Joseph constructed in \cite{Jos1,Jos2}, 
such a slice for certain ``truncated'' biparabolic subalgebras in a semisimple Lie algebra. 

In this paper, we give a simple explicit construction of an affine slice for the coadjoint action of a class of
(non truncated) biparabolic subalgebras of $\mathfrak{g}$ which includes all Borel subalgebras. The construction of the
affine subspace, and the proof that it is indeed an affine slice, are pretty straightforward, and they rely on some rather 
nice properties of Kostant's cascade construction of pairwise strongly orthogonal roots.

\subsection{}\label{basics}
Let us first fix some notations.
We shall assume from now on that $\mathfrak{g}$ is semisimple. 
Let $\mathfrak{h}$ be a Cartan subalgebra  
of $\mathfrak{g}$, $R$ the root system of $\mathfrak{g}$ relative to 
$\mathfrak{h}$, and $\Pi$ a set of simple roots of $R$.  Denote by $R_{+}$
(resp. $R_{-}$) the corresponding set of positive (resp. negative) roots.
For $\alpha \in R$, we denote by $\mathfrak{g}^{\alpha}$ the corresponding root subspace, and
$X_{\alpha}$ a non zero element of $\mathfrak{g}^{\alpha}$.

For any subset $E$ of $R$, we set 
$$
\mathfrak{g}^{E} = \sum_{\alpha\in E} \mathfrak{g}^{\alpha}.
$$

Denote by $\kappa$ the Killing form of $\mathfrak{g}$. This form induces a linear isomorphism $\varphi$
between $\mathfrak{g}$ and $\mathfrak{g}^{*}$. For any $X\in\mathfrak{g}$,  the 
corresponding linear form, denoted by $\varphi_{X}$, verifies
$\varphi_{X}(Y)= \kappa (X,Y)$ for all $Y\in\mathfrak{g}$. 

Identifying $\mathfrak{h}^{*}$ as linear forms on $\mathfrak{g}$ which are zero on $\mathfrak{g}^{R}$, 
for any linear form $\lambda\in\mathfrak{h}^{*}$, we denote $h_{\lambda}$ the unique element
in $\mathfrak{h}$ verifying $\lambda (H) = \kappa (h_{\lambda} , H)$ for all $H\in\mathfrak{h}$.
We have
\begin{itemize}
\item[(i)] $h_{a_{1}\lambda_{1}+\cdots + a_{r}\lambda_{r}} = a_{1} h_{\lambda_{1}} + \cdots +
a_{r} h_{\lambda_{r}}$ for any $\lambda_{1},\dots ,\lambda_{r}\in\mathfrak{h}^{*}$ and $a_{1},\dots ,a_{r}\in\Bbbk$.
\item[(ii)] $\Bbbk h_{\alpha} = [\mathfrak{g}^{\alpha},\mathfrak{g}^{-\alpha}]$ for any $\alpha \in R$.
\end{itemize}

\subsection{}\label{seaweed}
For any subset $S$ of $\Pi$, we set 
$$
R^{S} = R \cap \mathbb{Z} S \ , \ 
R_{+}^{S} = R \cap \mathbb{N}S = R_{+} \cap \mathbb{Z} S \ , \
R_{-}^{S} = R^{S} \setminus R_{+}^{S} = R_{-} \cap \mathbb{Z} S
$$
respectively the set of roots, positive roots and negative roots of the subroot system 
generated by $S$.

Let $S$ and $T$ be subsets of $\Pi$.  Set $\Delta_{S,T} = R_{+}^{S} \cup R_{-}^{T}$, and 
$$
\mathfrak{q}_{S,T} = \mathfrak{h} \oplus \mathfrak{g}^{\Delta_{S,T}}.
$$
Since $\Delta_{S,T} = (R_{-}\cup R_{+}^{S}) \cap (R_{+}\cup R_{-}^{T})$, we see easily that
$\mathfrak{q}_{S,T}$ is the biparabolic subalgebra of $\mathfrak{g}$ associated to the pair of
weakly opposite parabolic subalgebras
$( \mathfrak{h} \oplus \mathfrak{g}^{R_{-}\cup R_{+}^{S}} ,  \mathfrak{h} \oplus \mathfrak{g}^{R_{+}\cup R_{-}^{T}})$.

It is well-known (see \cite{Panyushev} or \cite[Chapter 40]{TYbook}) that if $\mathfrak{q}$ is a biparabolic
subalgebra of $\mathfrak{g}$, then there exist $S,T$ such that $\mathfrak{q}$ is conjugated to $\mathfrak{q}_{S,T}$.

We shall therefore fix two subsets $S$ and $T$ of $\Pi$, and consider the biparabolic subalgebra $\mathfrak{q}_{S,T}$.
Denote by $Q_{S,T}$ the connected algebraic subgroup of $G$ whose Lie algebra is $\mathfrak{q}_{S,T}$.

We shall conserve the above notations in the rest of this paper.

\section{Properties of Kostant's cascade construction}

\subsection{}\label{cascade}
In this section, we recall some basic properties of Kostant's cascade construction of pairwise strongly orthogonal roots,
and prove a technical lemma related to biparabolic subalgebras. 

Let $S\subset \Pi$. We define a set $\mathcal{K}(S)$ by induction on the cardinal of $S$ as follows:
\begin{itemize}
\item[(i)] $\mathcal{K}(\emptyset) = \emptyset$.
\item[(ii)] If $S_{1}, \dots ,S_{r}$ are connected components (of the Dynkin diagram) of $S$, then 
$$
\mathcal{K}(S) = \mathcal{K}(S_{1}) \cup \cdots \cup  \mathcal{K}(S_{r}).
$$
\item[(iii)] If $S$ is connected, then there is a unique largest positive root $\varepsilon_{S}$ in
$R_{+}^{S}$ and 
$$
\mathcal{K}(S) = \{ S \} \cup  \mathcal{K}\left( \{ \alpha \in S ; \varepsilon_{S}  (h_{\alpha}) =0 \}\right).
$$
\end{itemize}
It is an immediate consequence of the definition that if $K$ and $L$ are distinct elements of $\mathcal{K}(S)$, then 
$\varepsilon_{K}$ and $\varepsilon_{L}$ are strongly orthogonal. In particular, we have 
$\varepsilon_{K}(h_{\varepsilon_{L}})=0$.

The following properties for $K,L \in \mathcal{K}(S)$ are direct consequences of the definition (see also \cite[Chapter 40]{TYbook}) :
\begin{itemize}
\item[(P1)] $K$ and $L$ are connected, and either $K\subset L$, $L\subset K$ or $K\cap L=\emptyset$. Moreover, if
$K\cap L=\emptyset$, then $K$ is strongly orthogonal to $L$. 
\item[(P2)] The set $K^{\bullet}= \{ \alpha \in K ; \varepsilon_{K} (h_{\alpha} ) \neq 0 \}$ is of cardinal $2$ if 
$K$ is of $A_{\ell}$, $\ell \geqslant 2$, and is of cardinal $1$ otherwise.  Note also that if $K$ is of type $A$, then
$K^{\bullet}$ is exactly the set of endpoints of the Dynkin diagram of $K$.  
\item[(P3)] The connected components of  $K\setminus K^{\bullet}$ are elements of $\mathcal{K}(S)$.
In particular, $K^{\bullet}\cap L^{\bullet}=\emptyset$ if $K\neq L$.
\end{itemize}

\begin{lemma}\label{2cascades}
Let $S,T\subset \Pi$. 
\begin{itemize}
\item[a)] If $K\in \mathcal{K}(S) \cap \mathcal{K}(T)$ and $L\in \mathcal{K}(S)\cup \mathcal{K}(T)$, then
$\varepsilon_{K} \pm \varepsilon_{L}\not\in R$.
\item[b)] Suppose that the following conditions are verified :
\begin{itemize}
\item[i)] $S\cap T\neq \emptyset$, $\mathcal{K}(S)\cap \mathcal{K}(T) =\emptyset$;
\item[ii)] $\{ \varepsilon_{E} ; E\in \mathcal{K}(S) \cup \mathcal{K}(T) \}$ is a linearly independent set of roots.
\end{itemize}
For any non zero element $h\in \mathrm{Vect}(h_{\alpha} ; \alpha \in S\cap T)$, there exists $E\in \mathcal{K}(S)\cup \mathcal{K}(T)$
such that $\varepsilon_{E}(h)\neq 0$. 
\end{itemize}
\end{lemma}
\begin{proof}
Part a) is a direct consequence of the definition of the cascade construction. Part b) requires more work, and we shall prove it 
in several steps. 

{\bf Step 1.} Let us fix a non zero element $h\in  \mathrm{Vect} ( h_{\alpha} ; \alpha \in S\cap T )$. 
Since $S\cap T \subset \Pi$, we have the unique decomposition
$$
h = \sum_{\alpha\in S\cap T} \lambda_{\alpha} h_{\alpha}
$$
where $\lambda_{\alpha}\in\Bbbk$. Since $h\neq 0$, the set 
$C = \{ \alpha \in S\cap T ; \lambda_{\alpha}\neq 0\}$ is non empty.
Let us denote by $C_{1},\dots ,C_{r}$ the connected components of $C$. 

For $1\leqslant i \leqslant r$, we have $C_{i} \subset S\cap T \subset S$. 
By (P1), there is a unique minimal (by inclusion) element $S_{i}$ of $\mathcal{K}(S)$ 
containing $C_{i}$. It follows from (P3) and the fact that $S_{i}$ is minimal that the intersection $C_{i} \cap S_{i}^{\bullet}$ 
is non empty. Similarly, there is a unique minimal element $T_{i}$ of  $\mathcal{K}(T)$ containing $C_{i}$, and 
$C_{i} \cap T_{i}^{\bullet}$ is non empty.

Let $\overline{S} = \bigcup\limits_{i=1}^{r} S_{i}$ and $\overline{T} = \bigcup\limits_{i=1}^{r} T_{i}$. By
(P1), we have
$$
\mathcal{K}(\overline{S}) = \bigcup_{i=1}^{r} \mathcal{K}(S_{i})  \subset \mathcal{K}(S) \ , \
\mathcal{K}(\overline{T}) = \bigcup_{i=1}^{r} \mathcal{K}(T_{i}) \subset \mathcal{K}(T)$$
and $C \subset \overline{S} \cap \overline{T} \subset S\cap T$. Moreover, 
$\overline{S}$ and $\overline{T}$ verify the hypotheses of the lemma.

We may therefore assume that $S=\overline{S}$, $T=\overline{T}$. Furthermore, we may clearly also assume that
$S\cup T=\Pi$ and $\Pi$ is connected.

Under these assumptions, $S = \bigcup\limits_{i=1}^{r} S_{i}$, and it follows from (P1) that connected components
of $S$ are exactly the maximal elements (by inclusion) of $\{ S_{1}, \dots , S_{r}\}$.  Again, the same applies 
for $T$.

{\bf Step 2.} Suppose that $S_{i}$ is a connected component of $S$. Then
$$
\varepsilon_{S_{i}} ( h ) = \sum_{\alpha\in C} \lambda_{\alpha} \varepsilon_{S_{i}} (h_{\alpha}) 
=\sum_{\alpha\in C\cap S_{i}} \lambda_{\alpha} \varepsilon_{S_{i}} (h_{\alpha}) 
=\sum_{\alpha\in C\cap S_{i}^{\bullet}} \lambda_{\alpha} \varepsilon_{S_{i}} (h_{\alpha}) .
$$
If $C\cap S_{i}^{\bullet} = \{ \alpha \}$, then 
$$
\varepsilon_{S_{i}} ( h ) =\lambda_{\alpha} \varepsilon_{S_{i}} ( h_{\alpha} )\neq 0,
$$
and we have the result.  We saw in 1) that $C_{i} \cap S_{i}^{\bullet}$ is non empty.  So we are
reduced to the case where $C \cap S_{i}^{\bullet}$ is of cardinal $2$. 

Of course, this applies to all connected components of $S$, and similarly for those of $T$.
It follows from (P2) that we may assume that :
\begin{itemize}
\item[i)] Any connected component of $S$ (resp. $T$) is of type $A$ and has rank at least $2$.
\item[ii)] If $S_{i}$ (resp. $T_{i}$) is a connected component of $S$ (resp. $T$), then 
$S_{i}^{\bullet} \subset C$ (resp. $T_{i}^{\bullet}\subset C$).
\end{itemize}

We shall show that assumptions i) and ii) imply that $\Pi$ is of type $A$. 
Since $\Pi = S\cup T$,  any element of $\Pi$ is contained in a connected component of $S$ or $T$.
We shall consider a connected component of $S$ or $T$ containing $\alpha$ specified below 
according to the Dynkin type of $\Pi$. 
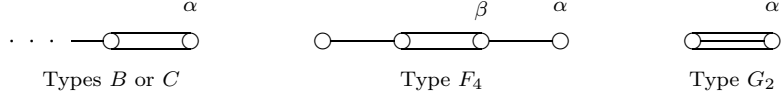
$$
\begin{array}{ccc}
\begin{Dynkin}
\Dbloc{\Ddots}
\Dbloc{\Dwest\Dcirc\Ddoubleeast}
\Dbloc{\Ddoublewest\Dcirc\Dtext{t}{\alpha}}
\end{Dynkin}
&
\begin{Dynkin}
\Dbloc{\Dcirc\Deast}
\Dbloc{\Dwest\Dcirc\Ddoubleeast}
\Dbloc{\Ddoublewest\Deast\Dcirc\Dtext{t}{\beta}}
\Dbloc{\Dwest\Dcirc\Dtext{t}{\alpha}}
\end{Dynkin}
& 
\begin{Dynkin}
\Dbloc{\Dcirc\Ddoubleeast\Deast}
\Dbloc{\Dcirc\Ddoublewest\Dwest\Dtext{t}{\alpha}}
\end{Dynkin}
\\ [-1em]
\hbox{\scriptsize  Types  $B$ or  $C$} & \hbox{\scriptsize Type $F_{4}$} &
\hbox{\scriptsize Type $G_{2}$}
\end{array}
$$
If $\Pi$ is of type $B$, $C$ or $G_{2}$, then it is not possible to have a connected component of $S$ or $T$ 
containing $\alpha$ specified below verifying the conditions of assumption i).  If $\Pi$ is of type $F_{4}$, then the only 
possible connected component of $S$ or $T$ containing $\alpha$ verifying the conditions of assummption i) is 
$\{ \alpha , \beta \}$. Now assumption ii) and (P2) say that $\{ \alpha ,\beta \} \subset C \subset S\cap T$. It follows that 
$\{ \alpha , \beta \}$ is a connected component for both $S$ and $T$, but this is impossible since 
$\mathcal{K}(S) \cap \mathcal{K}(T)=\emptyset$.

Now if $\Pi$ is of type $D$ or $E$ or $F$, we consider the connected components of $S$ or $T$ containing $\alpha$ 
and $\beta$ specified below. As in the case of type $F$, assumption i) implies that they must both be
the line joining $\alpha$ and $\beta$, and assumption ii) implies that $S$ and $T$ has  a common connected component.
Again this is impossible since 
$\mathcal{K}(S) \cap \mathcal{K}(T)=\emptyset$.
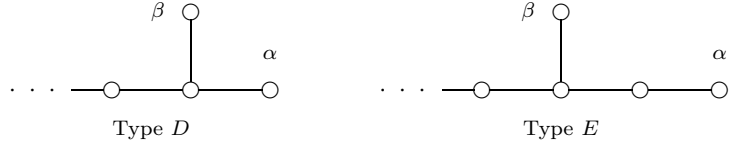
$$
\begin{array}{ccc}
\begin{Dynkin}
\Dspace\Dspace
\Dbloc{\Dcirc\Dsouth\Dtext{l}{\beta}}
\Dskip
\Dbloc{\Ddots}
\Dbloc{\Dwest\Dcirc\Deast}
\Dbloc{\Dwest\Dcirc\Deast\Dnorth}
\Dbloc{\Dwest\Dcirc\Dtext{t}{\alpha}}
\end{Dynkin}
&
\begin{Dynkin}
\Dspace\Dspace
\Dbloc{\Dcirc\Dsouth\Dtext{l}{\beta}}
\Dskip
\Dbloc{\Ddots}
\Dbloc{\Dwest\Dcirc\Deast}
\Dbloc{\Dwest\Dcirc\Deast\Dnorth}
\Dbloc{\Dwest\Dcirc\Deast}
\Dbloc{\Dwest\Dcirc\Dtext{t}{\alpha}}
\end{Dynkin}
\\ [-1em]
\hbox{\scriptsize Type $D$} & 
\hbox{\scriptsize Type $E$}
\end{array}
$$

We are therefore reduced to the case where $\Pi$ is of type $A$.

{\bf Step 3.} Let  us suppose that $\varepsilon_{E} ( h ) = 0$ for all $E\in \mathcal{K}(S) \cup \mathcal{K}(T)$. We shall show that this 
contradicts with the linear independence of the set $\{ \varepsilon_{E} ; E\in \mathcal{K}(S) \cup \mathcal{K}(T) \}$.

Let $S_{k}=\{ \beta_{1}, \dots ,\beta_{\ell} \}$ be a connected component of $S$ where the numbering 
of the $\beta$'s follows the numbering of the Dynkin diagram in \cite[Chapter 18]{TYbook}.
We have (see for example \cite[Chapter 40]{TYbook})
$$
\{ \varepsilon_{E} ; E\in \mathcal{K}(S_{k}) \}= \{ \varepsilon_{i} = \beta_{i} + \cdots + \beta_{\ell +1 -i} \ ; \ i = 1,\dots ,[\ell/2]\}
$$
where $[\ell/2]$ is the largest integer less than or equal to $\ell/2$.
For $1\leqslant i \leqslant [\ell/2]$, we have by (P1) that 
$$
0=\varepsilon_{i} ( h ) = \sum_{j=1}^{\ell} \lambda_{\beta_{j}} \varepsilon_{i} ( h_{\beta_{j}}) 
= \left\{ 
\begin{array}{cl}
c (\lambda_{\beta_{1}} + \lambda_{\beta_{\ell}}) &\hbox{if }  i=1; \\
c ( -\lambda_{\beta_{i-1}} + \lambda_{\beta_{i}} +\lambda_{\beta_{\ell +1 -i}} - \lambda_{\beta_{\ell + 2 -i }}) &\hbox{if }  i > 1
\end{array}
\right.
$$
where $c$ is a non zero constant because all the roots have the same length in type $A$.
It follows that $\lambda_{\beta_{j}} = -\lambda_{\beta_{\ell+1-j}}$ for all $j=1,\dots ,\ell$. 
In particular, $\lambda_{\beta_{j}} = 0$ if and only if $\lambda_{\beta_{\ell +1-j}}=0$. 

Since $S_{k}$ is a connected component of $S$, and $C\subset S\cap T$, any $C_{i}$ is either contained in $S_{k}$
or is strongly orthogonal to $S_{k}$. If $C_{i}$ is contained in $S_{k}$, then 
there exist $1\leqslant p\leqslant q\leqslant \ell$ such that 
$C_{i} = \{ \beta_{p} , \dots ,  \beta_{q} \}$. It follows from the discussion of the preceding paragraph that
$C_{i}^{\sigma} = \{Ê\beta_{\ell +1 -q }, \dots , \beta_{\ell+1-p} \} \subset C$. Moreover, since $\beta_{p-1}$ and $\beta_{q+1}$
are not in $C_{i}$, we deduce that $C_{i}^{\sigma}$ is a connected component of $C$. We have therefore a symmetry 
on the set of connected components of $C$ in $S_{k}$.

Since each connected component of $C$ is contained in a unique connected component of $S$, we obtain a permutation 
$\sigma$ on the set $\mathcal{C} = \{ C_{1} , \dots ,C_{r} \}$ by defining $\sigma (C_{i}) = C_{i}^{\sigma}$.

If $C_{i} \in \mathcal{C}$ is contained in $S_{k}$. By symmetry, 
we must have one of the following $2$ configurations :
$$
\mathrm{(I)} : \underbrace{\beta_{1} ,Ê\dots , \beta_{p} ,\underbrace{\overbrace{\beta_{p+1}, \dots , \beta_{q}}^{C}, \beta_{q+1} ,
\dots , \beta_{\ell-q} , \overbrace{\beta_{\ell+1-q} ,\dots , \beta_{\ell-p}}^{C'}}_{S_{i}} ,
\beta_{\ell+1-p} ,\dots ,\beta_{\ell}}_{S_{k}}
$$
where $(C,C')=(C_{i}, \sigma (C_{i}))$ or 
$(\sigma(C_{i}), C_{i})$, and 
$$
\mathrm{(II)} :
\underbrace{\beta_{1} ,Ê\dots , \beta_{p} ,\underbrace{\overbrace{\beta_{p+1}, \dots ,  \beta_{\ell-p}}^{C_{i}=\sigma (C_{i})}}_{S_{i}},
\beta_{\ell+1-p} ,\dots ,\beta_{\ell}}_{S_{k}}
$$
where in both configurations, $p$ can be equal to $0$ in which case $S_{i}=S_{k}$.
We deduce easily that we have
$$
\sum_{\alpha\in C_{i}} \alpha + \sum_{\alpha\in \sigma (C_{i}) } \alpha 
= \varepsilon_{p+1} - \varepsilon_{q+1} = \varepsilon_{S_{i}} - \varepsilon_{q+1}
$$
in configuration (I), and 
$$
\sum_{\alpha\in C_{i}} \alpha = \varepsilon_{p+1} = \varepsilon_{S_{i}}
$$
in configuration (II). Observe that 
\begin{itemize}
\item[(O1)] In configuration (I), by definition $\varepsilon_{q+1} = \varepsilon_{E}$ for some 
$E \in \mathcal{K}(S)$.
Moreover, since neither $\beta_{q+1}$ nor $\beta_{\ell-q}$ can be in 
$\mathrm{supp}(h)$, we have that $E\neq \varepsilon_{S_{m}}$  for any $m$. 

\item[(O2)] Since we are in the type $A$ case, we see from the above that if
$S_{j}=S_{i}$, then $C_{j} = C_{i}$ ou $\sigma (C_{i})$.
\end{itemize}

The same argument applied to $T$ gives another permutation $\tau$ and the properties above.
Observe that $\sigma$ and $\tau$ are both of order $2$.

{\bf Step 4.} We define a coloured graph $\mathcal{G}$ 
whose vertices are the elements of $\mathcal{C}$, and  
we have an edge coloured $\sigma$ (resp. $\tau$) between $C$ and $C'$ if $C'=\sigma (C)$ (resp. $C'=\tau(C)$).
So loops in $\mathcal{G}$ correspond exactly to elements of $\mathcal{C}$ fixed by either $\sigma$ or $\tau$.

The connected components of $\mathcal{G}$ correspond therefore to the orbits of $\mathcal{C}$ under the action of the group
generated by $\sigma$ and $\tau$.

Since $\sigma$ and $\tau$ are both of order $2$, any vertex $C$ is the endpoint of exactly one $\sigma$-coloured edge
and one $\tau$-coloured edge. It follows that a connected component of $\mathcal{G}$ has to be in one of the following
forms (after renumbering the $C_{i}$'s) : 
$$ 
\mathcal{G}_{1} :
\begin{picture}(30,12)
\put(0,0){\makebox(10,8)[l]{\scriptsize $\sigma$}}
\put(16,4){\oval(12,8)[l]}
\put(16,0){\line(4,1){10}}
\put(16,8){\line(4,-1){10}}
\end{picture} 
C_{1} 
\begin{picture}(30,12) 
\put(2,4){\line(1,0){26}}
\put(11,6){\makebox(10,6)[c]{\scriptsize$\tau$}}
\end{picture}
C_{2} 
\begin{picture}(30,12) 
\put(2,4){\line(1,0){26}}
\put(11,6){\makebox(10,6)[c]{\scriptsize$\sigma$}}
\end{picture}
\begin{picture}(20,12) 
\put(4,4){\circle*{0.9}}
\put(8,4){\circle*{0.9}}
\put(12,4){\circle*{0.9}}
\put(16,4){\circle*{0.9}}
\end{picture}
\begin{picture}(30,12) 
\put(2,4){\line(1,0){26}}
\put(11,6){\makebox(10,6)[c]{\scriptsize$\sigma$}}
\end{picture}
C_{k-1}
\begin{picture}(30,12) 
\put(2,4){\line(1,0){26}}
\put(11,6){\makebox(10,6)[c]{\scriptsize$\tau$}}
\end{picture}
C_{k} 
\begin{picture}(30,12)
\put(25,0){\makebox(20,8)[l]{\scriptsize $\sigma$}}
\put(14,4){\oval(12,8)[r]}
\put(14,0){\line(-4,1){10}}
\put(14,8){\line(-4,-1){10}}
\end{picture} 
$$
$$ 
\mathcal{G}_{2} :
\begin{picture}(30,12)
\put(0,0){\makebox(10,8)[l]{\scriptsize $\tau$}}
\put(16,4){\oval(12,8)[l]}
\put(16,0){\line(4,1){10}}
\put(16,8){\line(4,-1){10}}
\end{picture} 
C_{1} 
\begin{picture}(30,12) 
\put(2,4){\line(1,0){26}}
\put(11,6){\makebox(10,6)[c]{\scriptsize$\sigma$}}
\end{picture}
C_{2} 
\begin{picture}(30,12) 
\put(2,4){\line(1,0){26}}
\put(11,6){\makebox(10,6)[c]{\scriptsize$\tau$}}
\end{picture}
\begin{picture}(20,12) 
\put(4,4){\circle*{0.9}}
\put(8,4){\circle*{0.9}}
\put(12,4){\circle*{0.9}}
\put(16,4){\circle*{0.9}}
\end{picture}
\begin{picture}(30,12) 
\put(2,4){\line(1,0){26}}
\put(11,6){\makebox(10,6)[c]{\scriptsize$\tau$}}
\end{picture}
C_{k-1}
\begin{picture}(30,12) 
\put(2,4){\line(1,0){26}}
\put(11,6){\makebox(10,6)[c]{\scriptsize$\sigma$}}
\end{picture}
C_{k} 
\begin{picture}(30,12)
\put(25,0){\makebox(20,8)[l]{\scriptsize $\tau$}}
\put(14,4){\oval(12,8)[r]}
\put(14,0){\line(-4,1){10}}
\put(14,8){\line(-4,-1){10}}
\end{picture} 
$$
$$ 
\mathcal{G}_{3} :
\begin{picture}(30,12)
\put(0,0){\makebox(10,8)[l]{\scriptsize $\sigma$}}
\put(16,4){\oval(12,8)[l]}
\put(16,0){\line(4,1){10}}
\put(16,8){\line(4,-1){10}}
\end{picture} 
C_{1} 
\begin{picture}(30,12) 
\put(2,4){\line(1,0){26}}
\put(11,6){\makebox(10,6)[c]{\scriptsize$\tau$}}
\end{picture}
C_{2} 
\begin{picture}(30,12) 
\put(2,4){\line(1,0){26}}
\put(11,6){\makebox(10,6)[c]{\scriptsize$\sigma$}}
\end{picture}
\begin{picture}(20,12) 
\put(4,4){\circle*{0.9}}
\put(8,4){\circle*{0.9}}
\put(12,4){\circle*{0.9}}
\put(16,4){\circle*{0.9}}
\end{picture}
\begin{picture}(30,12) 
\put(2,4){\line(1,0){26}}
\put(11,6){\makebox(10,6)[c]{\scriptsize$\tau$}}
\end{picture}
C_{k-1}
\begin{picture}(30,12) 
\put(2,4){\line(1,0){26}}
\put(11,6){\makebox(10,6)[c]{\scriptsize$\sigma$}}
\end{picture}
C_{k} 
\begin{picture}(30,12)
\put(25,0){\makebox(20,8)[l]{\scriptsize $\tau$}}
\put(14,4){\oval(12,8)[r]}
\put(14,0){\line(-4,1){10}}
\put(14,8){\line(-4,-1){10}}
\end{picture} 
$$
$$ 
\mathcal{G}_{4} :
\begin{array}{c}
\hskip1pt
\begin{picture}(100,22)
\put(95,18){\line(-4,-1){80}}
\put(50,10){\makebox(10,6)[c]{\scriptsize$\sigma$}}
\end{picture} 
\begin{picture}(12,22)
\put(0,18){\makebox(10,0)[c]{$C_{1}$}}
\end{picture}
\begin{picture}(100,20)
\put(0,18){\line(4,-1){80}}
\put(40,10){\makebox(10,6)[c]{\scriptsize$\tau$}}
\end{picture} 
\\  [-0.3em]
C_{2} 
\begin{picture}(30,12) 
\put(2,4){\line(1,0){26}}
\put(11,6){\makebox(10,6)[c]{\scriptsize$\tau$}}
\end{picture}
C_{3} 
\begin{picture}(30,12) 
\put(2,4){\line(1,0){26}}
\put(11,6){\makebox(10,6)[c]{\scriptsize$\sigma$}}
\end{picture}
\begin{picture}(20,12) 
\put(4,4){\circle*{0.9}}
\put(8,4){\circle*{0.9}}
\put(12,4){\circle*{0.9}}
\put(16,4){\circle*{0.9}}
\end{picture}
\begin{picture}(30,12) 
\put(2,4){\line(1,0){26}}
\put(11,6){\makebox(10,6)[c]{\scriptsize$\tau$}}
\end{picture}
C_{2k-1}
\begin{picture}(30,12) 
\put(2,4){\line(1,0){26}}
\put(11,6){\makebox(10,6)[c]{\scriptsize$\sigma$}}
\end{picture}
C_{2k} 
\end{array}
$$

We now associated to an edge $\mathbf{e}$, a weight $\gamma_{\mathbf{e}}$ as follows : 
$$
\gamma_{\mathbf{e}} = \left\{
\begin{array}{ll}
\sum\limits_{\alpha\in C_{j}} \alpha + \sum\limits_{\alpha\in C_{l}} \alpha & 
\hbox{if } \mathbf{e} \hbox{ is an edge of the form } 
C_{j}  
\begin{picture}(30,12) 
\put(2,4){\line(1,0){26}}
\put(11,6){\makebox(10,6)[c]{\scriptsize$ $}}
\end{picture}
C_{l}, j\neq l\\
\sum\limits_{\alpha\in C_{j}} \alpha &
\hbox{if } \mathbf{e} \hbox{ is an edge of the form } 
\begin{picture}(30,12)
\put(0,0){\makebox(10,8)[l]{\scriptsize $ $}}
\put(16,4){\oval(12,8)[l]}
\put(16,0){\line(4,1){10}}
\put(16,8){\line(4,-1){10}}
\end{picture} 
C_{j}.
\end{array}
\right.
$$

For $1\leqslant i\leqslant 4$, let $\mathcal{G}_{i}^{\sigma}$ (resp. $\mathcal{G}_{i}^{\tau}$) denote
the set of $\sigma$-coloured edges (resp. $\tau$-coloured edges) in $\mathcal{G}_{i}$.
Since each vertex is the endpoint of exactly one $\sigma$-coloured edge and one $\tau$-coloured edge, 
we check easily from the forms above that
$$
\sum_{\mathbf{e}\in \mathcal{G}_{i}^{\sigma}} \gamma_{\mathbf{e}}
- 
\sum_{\mathbf{e}\in \mathcal{G}_{i}^{\tau}} \gamma_{\mathbf{e}}
=0.
$$

In view of Step 3, the weight of an edge can be expressed by a linear combination of elements of 
$\varepsilon_{E}$, $E\in \mathcal{K}(S) \cup \mathcal{K}(T)$. The above equality therefore gives a 
linear relation 
$$
\sum_{E\in \mathcal{K}(S) \cup \mathcal{K}(T)}  \mu_{E} \varepsilon_{E} = 0.
$$

By (O1), (O2) of Step  3 and the fact that $\mathcal{K}(S)\cap \mathcal{K}(T)=\emptyset$, 
the element $\varepsilon_{S_{1}}$ appears only in the expression of the weight in terms of 
$\varepsilon_{E}$, $E\in \mathcal{K}(S) \cup \mathcal{K}(T)$, 
of the edge 
$$
\begin{picture}(30,12)
\put(0,0){\makebox(10,8)[l]{\scriptsize $\sigma$}}
\put(16,4){\oval(12,8)[l]}
\put(16,0){\line(4,1){10}}
\put(16,8){\line(4,-1){10}}
\end{picture} 
C_{1} \hbox{ for } \mathcal{G}_{1} \hbox{ and } \mathcal{G}_{3},
$$
and  the edge 
$$
C_{1}  
\begin{picture}(30,12) 
\put(2,4){\line(1,0){26}}
\put(11,6){\makebox(10,6)[c]{\scriptsize$\sigma$}}
\end{picture}
C_{2} \hbox{ for } \mathcal{G}_{2} \hbox{ and } \mathcal{G}_{4}.
$$ 
We conclude that this relation between the $\varepsilon_{E}$'s is not trivial. This contradicts our 
hypothesis on the linear independence of  the set $\{ \varepsilon_{E}  ;  E\in \mathcal{K}(S) \cup \mathcal{K}(T)\}$.
The proof is now complete.  
\end{proof}

\subsection{}\label{counterexample}
Condition ii) in part b) of Lemma \ref{2cascades} can not be dropped. For example, if we take 
$\mathfrak{g}$ simple of type $A_{5}$, $S=\Pi$ and $T=\Pi \setminus \{ \alpha_{3}\}$ where the numbering of 
the simple roots is as in \cite[Chapter 18]{TYbook}. Then condition i) is verified while condition ii) is not 
verified, and the element 
$$
h=h_{\alpha_{1}} -h_{\alpha_{2}} +h_{\alpha_{4}} - h_{\alpha_{5}}
$$ 
verifies  $\varepsilon_{E}(h)=0$ for all $E\in \mathcal{K}(S)\cup \mathcal{K}(T)$.

\section{Slices for the coadjoint action of  biparabolic subalgebras}\label{section:slices}

\subsection{}\label{decomposition}
In the section, we fix subsets $S$, $T$ of $\Pi$. We conserve the notations of paragraphs \ref{seaweed} and \ref{cascade}.

Set 
$$
\begin{array}{c} 
\Gamma = \{ \varepsilon_{K} ; K\in\mathcal{K}(S) \} \cup  \{ -\varepsilon_{L} ; L\in\mathcal{K}(T) \}, \\ [0.3em]
\Gamma_{0} = R_{+} \cap (\Gamma \cap -\Gamma ) \ ,\  \Gamma_{1} = \Gamma \setminus 
(\Gamma_{0}\cup -\Gamma_{0}).
\end{array}
$$ 

Let $\Gothicd$ be the subspace of $\mathfrak{h}^{*}$ spanned by $\Gamma$, 
$\mathfrak{m} = \mathfrak{g}^{\Delta_{S,T} \cap \Gothicd}$, 
$\mathfrak{n} = \mathfrak{g}^{\Delta_{S,T} \setminus \Gothicd}$.
From \ref{seaweed}, we deduce immediately that  
\begin{equation} \label{eqn:q}
\mathfrak{q}_{S,T}= \mathfrak{h} \oplus \mathfrak{m} \oplus \mathfrak{n} \ , \ 
[\mathfrak{h},\mathfrak{m}] \subset \mathfrak{m} \ , \
[\mathfrak{h} + \mathfrak{m}, \mathfrak{n}] \subset \mathfrak{n} \ , \
[\mathfrak{m},\mathfrak{m}] \subset \mathfrak{m} \oplus \sum_{\alpha\in \Gamma} \Bbbk h_{\alpha}.
\end{equation}
In particular, $\mathfrak{h}\oplus \mathfrak{m}$ is a Lie subalgebra of $\mathfrak{q}_{S,T}$.

Let us identify $\mathfrak{q}_{S,T}^{*}$ with $\mathfrak{h}^{*}\oplus \mathfrak{m}^{*} \oplus \mathfrak{n}^{*}$
via the linear isomorphism $\varphi$ in paragraph \ref{basics}. In particular, $\mathfrak{m}^{*}$ (resp. $\mathfrak{n}^{*}$)
is the vector subspace spanned by $\varphi_{X_{-\alpha}}$, $\alpha \in \Delta_{S,T} \cap \Gothicd$
(resp. $\alpha \in \Delta_{S,T} \setminus \Gothicd$).
 
We deduce from the identities \eqref{eqn:q} above that 
\begin{equation} \label{eqn:q*}
\begin{array}{c}
\mathfrak{h}.\mathfrak{h}^{*}=\{ 0\} \ , \ 
\mathfrak{h}.\mathfrak{m}^{*} \subset \mathfrak{m}^{*} \ , \
\mathfrak{h}.\mathfrak{n}^{*} \subset \mathfrak{n}^{*} \ , \ 
\mathfrak{m}.\mathfrak{h}^{*} \subset \mathfrak{m}^{*} \ , \ 
\\ [0.3em]
\mathfrak{m}.\mathfrak{m}^{*} \subset \mathfrak{m}^{*}\oplus \Gothicd \ , \
\mathfrak{m}.\mathfrak{n}^{*} \subset \mathfrak{n}^{*} \ , \
\mathfrak{n}.\mathfrak{h}^{*} \subset \mathfrak{n}^{*} \ , \
\mathfrak{n}.\mathfrak{m}^{*} \subset \mathfrak{n}^{*} \ . \
\end{array}
\end{equation}

\begin{lemma}\label{lemma:e}
Let $\Gothice$ be the set of elements $\lambda$ of $\mathfrak{h}^{*}$ such that 
$\lambda (h_{\alpha})=0$ for all $\alpha \in \Gamma$. Then 
$$
\mathfrak{m}.\Gothice = \{ 0\} \ \hbox{ and } \ 
\mathfrak{h}^{*} = \Gothicd \oplus \Gothice.
$$
\end{lemma}
\begin{proof}
The first equality follows from \eqref{eqn:q}, \eqref{eqn:q*} and the definition of $\Gothice$. 
The second is a special case of  \cite[proposition 4, p.145]{Bourbaki}.
\end{proof}

\subsection{}\label{slice}
For $\mathbf{a} = (a_{\alpha})_{\alpha\in \Gamma}\in (\Bbbk\setminus \{ 0\})^{\Gamma}$, set 
$$
f_{\mathbf{a}}^{\alpha} = \left\{
\begin{array}{lll}
a_{-\alpha} \varphi_{X_{-\alpha}} & \hbox{if } \ \alpha \in \Gamma_{1}, \\
a_{\alpha} \varphi_{X_{\alpha}} + a_{-\alpha} \varphi_{X_{-\alpha}} & \hbox{if } \ \alpha \in \Gamma_{0}, 
\end{array}
\right.
$$
and $f_{\mathbf{a}} = \sum\limits_{\alpha\in \Gamma_{0}\cup \Gamma_{1}} f_{\mathbf{a}}^{\alpha}$,
all considered as linear forms on $\mathfrak{q}_{S,T}$.

Let 
$$
\mathfrak{m}_{0}^{*} = \sum_{\alpha\in \Gamma_{0}} \Bbbk f_{\mathbf{a}}^{\alpha} \ , \
W_{\mathbf{a}} = \Gothice \oplus \mathfrak{m}_{0}^{*} \ \hbox{ and } \
V_{\mathbf{a}} = f_{\mathbf{a}} + W_{\mathbf{a}}.
$$
Note that $f_{\mathbf{a}}\in\mathfrak{m}^{*}$ and $W_{\mathbf{a}} \subset \mathfrak{h}^{*}\oplus \mathfrak{m}_{0}^{*}$.

\begin{lemma}\label{C2}
For all $f\in V_{\mathbf{a}}$, we have $\mathfrak{q}_{S,T}.f\cap W_{\mathbf{a}} =\{0\}$. In particular,
$\mathrm{T}_{f}(Q_{S,T}.f) \cap T_{f}(V_{\mathbf{a}}) = \{0\}$ for all $f\in V$.
\end{lemma}
\begin{proof}
Let $(H,X,Y)\in\mathfrak{h}\times \mathfrak{m} \times \mathfrak{n}$ and $(t,m) \in \Gothice \times \mathfrak{m}_{0}^{*}$ 
be such that 
$$
(H+X+Y). (f_{\mathbf{a}}+t+m) = w\in W_{\mathbf{a}} \subset \mathfrak{h}^{*}\oplus \mathfrak{m}_{0}^{*}.
$$
From the identities \eqref{eqn:q*}, we deduce that $Y.(f_{\mathbf{a}}+t+m)=0$, $H.t=X.t=0$, 
$H.(f_{\mathbf{a}}+m) \in \mathfrak{m}^{*}$, $X.(f_{\mathbf{a}}+m)\in \mathfrak{m}^{*}\oplus \Gothicd$,
and therefore
$$
w = (H+X).(f_{\mathbf{a}}+ m) \in ( \mathfrak{m}^{*}\oplus \Gothicd )\cap W_{\mathbf{a}} = \mathfrak{m}_{0}^{*}
$$
by Lemma \ref{lemma:e}.

Since $H.(f_{\mathbf{a}}+m)$ is a linear combination of $\varphi_{X_{-\alpha}}$, $\alpha \in \Gamma$,
we deduce that 
$$
X.(f_{\mathbf{a}}+m) = m_{0} + m_{1} 
$$
where $m_{0}$ is a linear combination of $\varphi_{X_{\alpha}}$, $\alpha \in \Gamma_{0}\cup -\Gamma_{0}$, and
$m_{1}$ is a linear combination of $\varphi_{X_{-\alpha}}$, $\alpha\in \Gamma_{1}$.

On the other hand, $X.(f_{\mathbf{a}}+m)$ is a linear combination of elements of the form 
$X_{\alpha}\varphi_{X_{-\beta}} \in \Bbbk \varphi_{X_{\alpha-\beta}}$ 
with $\alpha \in \Delta \cap \Gothicd$ and $\beta\in\Gamma$.
If $m_{0}$ is non-zero, then we can find $\alpha\in \Delta\cap \Gothicd$ and $\beta\in\Gamma$ 
such that $\alpha-\beta \in \Gamma_{0}\cup -\Gamma_{0}$. This is impossible by  Lemma 
\ref{2cascades} a). Thus $m_{0}=0$, and we have
$$
w = (H+X).(f_{\mathbf{a}}+ m) = H.(f_{\mathbf{a}}+ m)  + m_{1}.
$$

Finally, for $\alpha \in \Gamma_{0}$, we have
$$
H.f_{\mathbf{a}}^{\alpha} = \alpha (H) \big( a_{\alpha} \varphi_{X_{\alpha}} - a_{-\alpha} \varphi_{X_{-\alpha}} \big).
$$
The elements $f_{\mathbf{a}}^{\alpha}$, $H.f_{\mathbf{a}}^{\alpha}$, $\alpha\in \Gamma_{0}$,
are linearly independent since $\mathbf{a}$ has non zero entries. It follows that $\alpha (H)=0$
for all $\alpha \in \Gamma_{0}$, and so 
$H.(f_{\mathbf{a}}+m)$ is a linear combination of $\varphi_{X_{-\alpha}}$, $\alpha \in \Gamma_{1}$.
We deduce that $w$ is also a linear combination of $\varphi_{X_{-\alpha}}$, $\alpha \in \Gamma_{1}$.
Since $w\in \mathfrak{m}_{0}^{*}$, we have $w=0$ as required.
\end{proof}

For $\alpha\in \Gamma_{0}$, set 
$$
Z_{\alpha} = a_{\alpha} X_{\alpha} + a_{-\alpha}X_{-\alpha}.
$$
Let $\mathfrak{t} =\{ H\in \mathfrak{h} ; \alpha (H)=0$ for all $\alpha \in \Gamma\}$, and 
$$
\mathfrak{r}_{\mathbf{a}} = \mathfrak{t} \oplus \sum_{\alpha\in \Gamma_{0}} \Bbbk Z_{\alpha}
\subset \mathfrak{t} \oplus \mathfrak{g}^{\Gamma_{0}\cup -\Gamma_{0}}.
$$
\begin{lemma}\label{centralizer}
We have $\mathfrak{r}_{\mathbf{a}}.W_{\mathbf{a}} = \{ 0 \}$, and 
$\mathfrak{r}_{\mathfrak{a}} \subset \mathfrak{q}_{S,T}^{f}$ for all $f\in V_{\mathbf{a}}$
where $\mathfrak{q}_{S,T}^{f} =\{ X\in \mathfrak{q}_{S,T} ; X.f = 0 \}$.
\end{lemma}
\begin{proof}
Observe that for $\alpha \in \Gamma_{0}$, we have $f_{\mathbf{a}}^{\alpha}=\varphi_{Z_{\alpha}}$, and a 
one verifies easily using Lemma \ref{2cascades} a)  that 
$$
Z_{\alpha}.f_{\mathbf{a}}^{\beta} =0
$$
for any $\beta \in \Gamma$. In particular, $\mathfrak{r}_{\mathbf{a}} \subset \mathfrak{q}_{S,T}^{f_{\mathbf{a}}}$.

By Lemma \ref{lemma:e}, $\mathfrak{m}.\Gothice = \{ 0\}$, and $\mathfrak{t}.\Gothice = \{ 0\}$
by \eqref{eqn:q*}, the result follows easily.
\end{proof}

\begin{lemma}\label{C1}
Suppose that $\mathfrak{r}_{\mathbf{a}} = \mathfrak{q}_{S,T}^{f_{\mathbf{a}}}$. 
Then there exists an open subset $\mathcal{O}_\mathbf{a}$ of 
$V_{\mathbf{a}}$ such that $Q_{S,T}.\mathcal{O}_{\mathbf{a}}$ is dense in $\mathfrak{q}_{S,T}^{*}$.
\end{lemma}
\begin{proof}
By definition, $\mathfrak{r}_{\mathbf{a}}$ is a commutative Lie subalgebra whose elements are
all semisimple. It follows from  \cite[40.1.3, 40.1.5, 40.1.6]{TYbook} that $f_{\mathbf{a}}$ is a stable, and hence regular, 
element of $\mathfrak{q}_{S,T}^{*}$. In particular, the $Q_{S,T}$-orbit of $f_{\mathbf{a}}$ has maximal dimension, or 
equivalently  the dimension of its stabilizer $Q_{S,T}^{f_{\mathbf{a}}}$ is minimal.
Moreover, by our hypothesis,
$$
\dim Q_{S,T}^{f_{\mathbf{a}}} = \dim \mathfrak{q}_{S,T}^{f_{\mathbf{a}}} =
\dim  \mathfrak{r}_{\mathbf{a}}  = \dim \mathfrak{t} + \sharp \Gamma_{0} 
= \dim \Gothice + \sharp \Gamma_{0} = \dim W_{\mathbf{a}}.
$$
Let $\mathcal{O}_{\mathbf{a}}$ be the set of regular elements $\mathfrak{q}_{S,T}^{*}$ contained 
in $V_{\mathbf{a}}$. It is a non-empty open subset of $V_{\mathbf{a}}$. 

Consider the $Q_{S,T}$-equivariant morphism
$$
\Phi : Q_{S,T} \times \mathcal{O}_{\mathbf{a}} \longrightarrow \mathfrak{q}_{S,T}^{*} \ , \
(\sigma , f )\mapsto \sigma (f).
$$
Let $f\in \mathcal{O}_{\mathbf{a}}$.
Then $\Phi^{-1}(f) = \{ (\sigma ,\sigma^{-1} (f) ) ; \sigma^{-1} (f) \in \mathcal{O}_{\mathbf{a}} \}$.
By Lemma \ref{C2}, $Q_{S,T}.f \cap \mathcal{O}_{\mathbf{a}}$ is a finite set. It follows that 
$$
\dim \Phi^{-1} (f) = \dim Q_{S,T}^{f} = \dim Q_{S,T}^{f_{\mathbf{a}}}=\dim W_{\mathbf{a}}, 
$$
and hence
$$
\dim \mathfrak{q}_{S,T} + \dim W_{\mathbf{a}} = 
\dim (Q_{S,T} \times \mathcal{O}_{\mathbf{a}} )= \dim \Phi^{-1} (f) + \dim \mathrm{im}(\Phi).
$$
We deduce that $\dim \mathrm{im}(\Phi )= \dim \mathfrak{q}_{S,T}=\dim \mathfrak{q}_{S,T}^{*}$, thus 
$\Phi$ is a dominant morphism, and the result follows.
\end{proof}

\begin{theorem}\label{result}
Suppose that $\Gamma = \Gamma_{1}$ is a linearly independent subset of roots.  
Then there exists $\mathbf{a}\in (\Bbbk \setminus \{ 0\})^{\Gamma}$ such that $V_{\mathbf{a}}$ is 
a slice for the coadjoint action of $\mathfrak{q}_{S,T}$.
\end{theorem}
\begin{proof}
By \cite[40.9.4]{TYbook}, there exists $\mathbf{a}\in (\Bbbk \setminus \{ 0\})^{\Gamma}$ such that
$f_{\mathbf{a}}$ is a stable element of $\mathfrak{q}_{S,T}^{*}$. Moreover, since $\Gamma_{0}$ is empty,
we have $\mathfrak{q}_{S,T}^{f_{\mathbf{a}}} = \mathfrak{t} = \mathfrak{r}_{\mathbf{a}}$.
By Lemmas \ref{C1} and \ref{C2}, we only need to show that 
condition $(\mathrm{C}_{3})$ is verified for some open subset of $\mathcal{O}_{\mathbf{a}}$
where $\mathcal{O}_{\mathbf{a}}$ is the open subset of $V_{\mathbf{a}}$
defined in the proof of Lemma \ref{C1}.

In view of Lemma \ref{centralizer}, $\mathcal{O}_{\mathbf{a}}$ is the set of elements 
$f$ in $V_{\mathbf{a}}$ verifying $\mathfrak{q}_{S,T}^{f} = \mathfrak{t}$. Now suppose that
$f\in \mathcal{O}_{\mathbf{a}}$ and $\sigma \in Q_{S,T}$ be such that $\sigma (f) \in \mathcal{O}_{\mathbf{a}}$.
Since $\mathfrak{q}_{S,T}^{f} = \mathfrak{q}_{S,T}^{\sigma (f)} = \sigma ( \mathfrak{q}_{S,T}^{f} )$, we deduce 
that $\sigma \in N_{Q_{S,T}} (\mathfrak{t})$.

Denote $\Delta_{0}= \Delta \cap -\Delta$, $\mathfrak{s} = \mathfrak{h} \oplus \mathfrak{g}^{\Delta_{0}}$,
$\mathfrak{l} = [\mathfrak{s},\mathfrak{s}]$, 
$\mathfrak{u} = \mathfrak{g}^{\Delta\setminus \Delta_{0}}$, and 
$\mathfrak{a}$ the centre of the reductive subalgebra $\mathfrak{s}$. 
Then $\mathfrak{q}_{S,T} = \mathfrak{l} \oplus \mathfrak{a} \oplus \mathfrak{u}$
is a (refined) Levi decomposition of $\mathfrak{q}_{S,T}$, where 
$\mathfrak{k} = \mathfrak{a} \oplus \mathfrak{u}$
is the radical of $\mathfrak{q}_{S,T}$, and $\mathfrak{l}$ is a Levi subalgebra of $\mathfrak{q}_{S,T}$.

In particular, we have $\mathfrak{t} \subset \mathfrak{h} \subset \mathfrak{l} \oplus \mathfrak{a}$,
and  $[\mathfrak{q}_{S,T}, \mathfrak{q}_{S,T} ] =  \mathfrak{l} \oplus \mathfrak{u}$. We deduce easily that 
$$
\mathfrak{t}\cap [\mathfrak{q}_{S,T}, \mathfrak{q}_{S,T} ]  = \mathfrak{t}\cap \mathfrak{l}
=\mathfrak{t} \cap \mathrm{Vect}(h_{\alpha} ; \alpha \in S\cap T).
$$
Thus $\mathfrak{t}\cap \mathfrak{l}=\{ 0\}$ if $S\cap T=\emptyset$. If $S\cap T\neq \emptyset$, then 
the fact that $\Gamma = \Gamma_{1}$ is a linearly independent subset implies that conditions i) and ii) of 
Lemma \ref{2cascades} b) are verified, and it follows immediately from the conclusion of Lemma \ref{2cascades} b) 
and the definition of $\mathfrak{t}$ that
\begin{equation}\label{eqn:intersection}
\mathfrak{t}\cap \mathfrak{l} = \mathfrak{t} \cap \mathrm{Vect}(h_{\alpha} ; \alpha \in S\cap T) = \{ 0\}.
\end{equation}
So we have $\mathfrak{t}\cap \mathfrak{l}=\{ 0\}$ in both cases. 

Let $L$ and $K$ be the connected algebraic subgroups of $Q_{S,T}$ whose Lie algebras are
$\mathfrak{l}$ and $\mathfrak{k}$ respectively. Then $Q_{S,T}=KL $. Let us write $\sigma = \sigma_{K}  \sigma_{L}$
where $\sigma_{K} \in K$ and $\sigma_{L} \in L$. 

Let $x \in\mathfrak{t}$. Then $x=x_{l} + x_{a}$ where $x_{l} \in \mathfrak{l}$ and $x_{a} \in \mathfrak{a}$.
Since $[\mathfrak{l}, \mathfrak{a}]=\{ 0\}$ and $[\mathfrak{q}_{S,T}, \mathfrak{q}_{S,T} ] =  \mathfrak{l} \oplus \mathfrak{u}$,
we have $\sigma_{L} (x_{a})=x_{a}$, and 
$$
\sigma (x) = 
\sigma_{K} (\sigma_{L} (x_{l} + x_{a} ) )
= \sigma_{K} (\underbrace{\sigma_{L} (x_{l})}_{\in \mathfrak{l}} ) + \sigma_{K} (x_{a}) 
=  \sigma_{L} (x_{l}) + y + x_{a} + z
$$
where $y, z\in \mathfrak{u}$. But $\sigma (x) \in \mathfrak{t}$, and so
$\sigma (x) = \sigma_{L} (x_{l}) + x_{a}$ and 
$$
\sigma (x) - x  = \sigma_{L} (x_{l}) -x_{l} \in \mathfrak{t} \cap \mathfrak{l} = \{ 0\}
$$
by \eqref{eqn:intersection}. Thus $\sigma (x) = x$. 

This being true for any $x\in \mathfrak{t}$, we deduce that $\sigma \in C_{Q_{S,T}}(\mathfrak{t})$ which 
is the connected algebraic subgroup of $Q_{S,T}$ whose Lie algebra is $C_{\mathfrak{q}_{S,T}}(\mathfrak{t})$.

By definition, we have 
$$
C_{\mathfrak{q}_{S,T}}(\mathfrak{t}) = \mathfrak{h} \oplus \mathfrak{m}.
$$
Let $X\in C_{\mathfrak{q}_{S,T}}(\mathfrak{t})$, then by Lemma \ref{lemma:e} and 
the identities \eqref{eqn:q*}, we have
$$
X.\Gothice =\{ 0\} \ , \ X.f_{\mathbf{a}} \in \mathfrak{m}^{*} \oplus \Gothicd.
$$
It follows that $\sigma (g) = g$ for all $g\in \Gothice$, and 
$\sigma (f_{\mathbf{a}}) - f_{\mathbf{a}} \in  \mathfrak{m}^{*} \oplus \Gothicd$.

Writing $f= f_{\mathbf{a}} + g$ where $g\in W_{\mathbf{a}}=\Gothice$ (since $\Gamma_{0}$ is empty),
we have 
$$
\sigma (f) - f = \sigma (f_{\mathbf{a}}) + g - f =
\sigma (f_{\mathbf{a}}) - f_{\mathbf{a}} \in ( \mathfrak{m}^{*} \oplus \Gothicd )\cap \Gothice 
=\{ 0\}.
$$
Hence $\sigma (f_{\mathbf{a}})=f_{\mathbf{a}}$, and $\sigma (f)=f$. So condition $(\mathrm{C}_{3})$
is verified by $\mathcal{O}_{\mathbf{a}}$.
\end{proof}

The hypothesis of Theorem \ref{result} is clearly satisfied when $S$ or $T$ is empty. We have the 
following result.

\begin{corollary}\label{slice:borel}
An affine slice exists for the coadjoint action of a Borel subaglebra.
\end{corollary}

\begin{remark}
When $\mathfrak{g}$ is simple of type $A_{\ell}$, then any (except one when $\ell$ is odd)
minimal parabolic subalgebra of $\mathfrak{g}$ verifies the hypotheses of Theorem \ref{result}.
So an affine slice exists for the coadjoint action of these minimal parabolic subalgebras.
\end{remark}

\begin{remark}
When $\Gamma$ is a linearly independent set with $\Gamma_{0}$ non empty, $V_{\mathbf{a}}$
is not in general an affine slice for any $\mathbf{a}$. Take for example $S=T=\Pi$, then $V_{\mathbf{a}}$
is a Cartan subalgebra and therefore condition $(\mathrm{C}_{3})$ is not verified. 
\end{remark}

\subsection{}
We finish the paper by establishing the claim in the introduction that the existence of an affine slice
for the coadjoint action of a Lie algebra $\mathfrak{g}$ implies that the field of $G$-invariant rational 
functions on $\mathfrak{g}^{*}$ is a purely transcendental extension of  $\Bbbk$. 

\begin{theorem}\label{extension}
Let $\mathcal{S}$ be an affine slice for the coadjoint action of a Lie algebra $\mathfrak{g}$, and 
denote by $G$ the algebraic adjoint group of $\mathfrak{g}$.
\begin{itemize}
\item[a)] The field of $G$-invariant rational fonctions on $\mathfrak{g}^{*}$ is a purely transcendental 
extension of $\Bbbk$. 
\item[b)] There exists an open subset $U$ of $\mathcal{S}$ such that 
the ring of regular functions on $U$ is isomorphic to the ring of $G$-invariant regular functions on $G.U$.
\end{itemize}
\end{theorem}
\begin{proof}
By Rosenlicht's Theorem \cite{Rosen}, there exists a non-empty $G$-stable open subset $\mathcal{U}$ of $\mathfrak{g}^{*}$
such that a geometric quotient $\mathcal{U}/G$ exists. Let us denote
$$
\pi : \mathcal{U} \rightarrow \mathcal{U}/G
$$
this quotient morphism. Recall that $\pi$ is an open morphism.

Let $\mathcal{O}$ be an open subset of $\mathcal{S}$ verifying the
conditions ($\mathrm{C}_{1}$), ($\mathrm{C}_{2}$), ($\mathrm{C}_{3}$). 
In particular, $\overline{G.\mathcal{O}} = \mathfrak{g}^{*}$. So $G.\mathcal{O}$ contains 
a non-empty $G$-stable open subset $\mathcal{W}$ of $V$. 

Set $\Omega=\mathcal{U}\cap \mathcal{W}$. Then $\Omega$ is a non-empty 
$G$-stable open subset of $\mathfrak{g}^{*}$, and $\Omega\cap \mathcal{S}$ is a non-empty open subset 
of $\mathcal{S}$ verifying
$$
G. (\Omega \cap \mathcal{S})=\Omega.
$$

Consider the morphism 
$$
\Phi : \Omega \cap \mathcal{S} \rightarrow \mathcal{U}/G \ , \  x\mapsto  \pi (x). 
$$
By our construction of $\Omega$, $\Phi$ is injective and its image
is $\pi (\Omega )$. It follows that $\Phi$ is dominant ($\pi$ being open). 

Being a non-empty open subset of an affine space, $\Omega\cap \mathcal{S}$ is normal, 
and hence by \cite[Corollary 17.4.4]{TYbook} and the injectivity of $\Phi$, we deduce that $\Phi$ is a birational
equivalence. Thus
we have the following isomorphism of rational functions 
$$
\mathbf{R}( \Omega \cap \mathcal{S} ) \simeq \mathbf{R}(\mathcal{U}/G) \simeq \mathbf{R}(\mathcal{U})^{G}  = 
\mathbf{R}(\mathfrak{g}^{*})^{G} 
$$
by \cite[Proposition 25.3.6]{TYbook}.

Since $\mathbf{R}(\Omega\cap \mathcal{S})=\mathbf{R}(\mathcal{S})$ is the field of fractions of 
a polynomial algebra over $\Bbbk$, we have part a).

Part b) is a direct consequence of the fact that $\Phi$ is a birational equivalence and $\pi$
is a geometric quotient.
\end{proof}

\begin{remark}
Of course, we may generalize the notion of an affine slice to any finite-dimensional $\mathfrak{g}$-module, and 
Theorem \ref{extension} remains valid for any $\mathfrak{g}$-module admitting an affine slice.
\end{remark}

\end{document}